\documentclass{article}

\usepackage{amsmath,amsthm,amssymb}
\usepackage{showlabels}

\newtheorem{lemma}{Lemma}
\newtheorem{theorem}{Theorem}
\theoremstyle{remark}

\newtheorem{remark}{Remark}

\newcommand{\KK}{C_{\infty}}
\newcommand{\K}{C}
\newcommand{\MM}{M_{\infty}}

\newcommand{\calo}{O}

\newcommand{\N}{\mathbb{N}}

\def\le{\leqslant}
\def\ge{\geqslant}

\author{Nikolay Vereshchagin\thanks{The article
  was prepared within the framework of the HSE University Basic Research Program and funded by the Russian Academic Excellence Project '5-100'. The author
  was in part funded by RFBR according to the research project № 19-01-00563.}\\
Moscow State University
and  HSE University, Russian Federation.}

\title{Descriptive Complexity of Computable Sequences Revisited}

\begin{document}
\maketitle

\begin{abstract}
  The purpose of this paper is
  to answer two questions left open in [B. Durand, A. Shen, and N. Vereshchagin, Descriptive Complexity of Computable Sequences, Theoretical Computer Science 171 (2001), pp. 47--58].
  Namely, we consider
  the following two
  complexities of an infinite computable 0-1-sequence $\alpha$:
  $C^{0'}(\alpha )$, defined as 
  the minimal length of a program with oracle $0'$ that prints
  $\alpha$,
  and  $\MM(\alpha)$,  defined as $\liminf C(\alpha_{1:n}|n)$, where
  $\alpha_{1:n}$ denotes the length-$n$ prefix of $\alpha$
  and $C(x|y)$ stands for conditional Kolmogorov complexity.
  We show that $C^{0'}(\alpha )\le \MM(\alpha)+O(1)$ and
  $\MM(\alpha)$ is not bounded by any computable function of $C^{0'}(\alpha )$,
  even
  on the domain of computable sequences.
\end{abstract}

\section{Introduction}

The notion of Kolmogorov complexity for finite binary 
strings was introduced in the 60ies independently by Solomonoff, 
Kolmogorov and Chaitin~\cite{sol,kolm,chaitin}.  There are different 
versions (plain Kolmogorov complexity, prefix complexity, etc.\ 
see~\cite{usp} for the details) that differ from each other not more 
than by an additive term logarithmic in the length of the argument. 
In the sequel we are using plain Kolmogorov complexity $C(x|y)$ as
defined in~\cite{kolm}, but similar results can be obtained for prefix
complexity.

When an infinite $0$-$1$-sequence is given, we may study the
complexity  of its finite prefixes. If prefixes have
high complexity, the sequence is random (see~\cite{lv,zl} for
details and references); if prefixes have low complexity, the
sequence is computable. In the sequel, we study  the
latter type.

Let $C(x)$, $C(x|y)$ denote the plain Kolmogorov complexity
 of a binary string $x$ and the conditional
Kolmogorov complexity of $x$ when $y$ (some other
binary string) is known. Let $\alpha_{1:n}$ denote first $n$ bits
(= length-$n$ prefix) of the sequence $\alpha$.
Let us recall the following criteria of computability of $\alpha$
in terms of complexity of its finite prefixes.

\begin{itemize}

\item[(a)] $\alpha$ is computable if and only if
$C(\alpha_{1:n}|n) =\calo(1)$.  This result is attributed in \cite{lov}
to A.R.~Meyer (see also~\cite{lv,zl}).

\item[(b)] $\alpha$ is computable if and only if
$C(\alpha_{1:n})<C(n)+\calo(1)$ \cite{chaitin2}.

\item[(c)] $\alpha$ is computable if and only if
$C(\alpha_{1:n})<\log_2n+\calo(1)$ \cite{chaitin2}.

\end{itemize}

These results provide criteria of the computability of infinite
sequences. For example, (a) can be reformulated as follows:
a sequence $\alpha$ is computable if and only if $M(\alpha)$ is finite,
where
     $$
M(\alpha )=\max_n C(\alpha _{1:n}|n)=\max_n\min_p \{l(p)\mid p(n)=\alpha _{1:n}\}.
     $$
Here $l(p)$ stands for the length of program $p$; $p(n)$ denotes its
output on $n$. As usual in Kolmogorov complexity theory,
we assume that some optimal programming language $U$ is fixed.
That is,
$(p,n)\mapsto U(p,n)$ is a computable function such that
for any other computable function $V(p,n)$ there is
a constant $c$ such that for all $p$ there is $p'$ with
$l(p')\le l(p)+c$ and $U(p',n)=V(p,n)$ for all $n$.
By $p(n)$ we then denote $U(p,n)$; conditional Kolmogorov
complexity is defined as $C(x|n)=\min\{l(p)\mid p(n)=x\}$
and unconditional Kolmogorov
complexity is defined as $C(x)=C(x|0)$.
(For more details see~\cite{lv,zl}.)

Therefore, $M(\alpha)$ can be considered as a complexity measure
of computable sequences. 
Another straightforward approach is to define complexity
of a sequence $\alpha $ as the length of the shortest program
computing $\alpha$:
   $$
C(\alpha )=\min\{l(p)\mid \forall n\ p(n)=\alpha _{1:n}\},
   $$
(and by definition $C(\alpha )=\infty$ if $\alpha $ is not computable.)

The difference between $C(\alpha)$ and $M(\alpha)$ can
be explained as follows: $M(\alpha)\le m$ means that for every $n$
there is a program $p_n$ of size at most $m$ that computes
$\alpha_{1:n}$ given $n$; this program may depend on $n$. On the
other hand, $C(\alpha )\le m$ means that there is a one such program
that works for all $n$. Thus, $M(\alpha)\le C(\alpha)$ for all $\alpha$, and
one can expect that $M(\alpha)$ may be significantly less than
$C(\alpha)$. (Note that the known proofs of (a) give no bounds of
$C(\alpha)$ in terms of $M(\alpha)$.)

Indeed, Theorem~3 from~\cite{bsv} shows that there is no computable
bound for $C(\alpha)$ in terms of $M(\alpha)$: for any computable
function $f(m)$ there exist computable infinite
sequences $\alpha^0, \alpha^1,\alpha^2\dots$ such that $M(\alpha^m)\le
m+\calo(1)$ and $C(\alpha^m)\ge \alpha(m)$.

The situation changes surprisingly when
we compare ``almost all''
versions of $C(\alpha )$ and $M(\alpha )$ defined in the following way:
\begin{align*}
\KK(\alpha )&=\min\{ l(p)\mid \forall^{\infty}n\ p(n)=\alpha _{1:n}\}\\
\MM(\alpha )&=\limsup_n C(\alpha _{1:n}|n)=\min\{m\mid\forall^{\infty}n \exists p\
(l(p)\le m \mbox{ and } p(n)=\alpha _{1:n})\},
\end{align*}
($\forall^{\infty}n$ stands for ``for all but finitely many $n$'').
It is easy to  see  that $\MM(\alpha)$ is finite only for computable
sequences. Indeed, if
$\MM(\alpha)$ is finite, then $M(\alpha)$ is also finite,  and
the computability  of $\alpha$
is implied by Meyer's theorem. All the four complexity
measures mentioned above are ``well calibrated'' in
the following sense: there are $\Theta(2^m)$ sequences whose complexity
does not exceed $m$. 

Surprisingly, it turns out that $\KK(\alpha)\le 2\MM(\alpha)+\calo(1)$
\cite[Theorem~5]{bsv} so the difference between $\KK$ and $\MM$ is
not so large as between $C$ and $M$. 
As this bound is tight: Theorem 6 from \cite{bsv}
proves that for every $m$ there is a sequence $\alpha$
with $\KK(\alpha)\ge 2m$ and $M(\alpha)\le m+O(1)$ (and hence $\MM(\alpha)\le m+O(1)$).

Finally, by Theorem 2 from~\cite{bsv},
$M(\alpha)$ cannot by bounded by any computable function 
of $\KK(\alpha)$ (and hence of $\MM(\alpha)$).

It is interesting also to compare $\KK$ and $\MM$ with
relativized versions of $C$. For any oracle $A$
one may consider a relativized Kolmogorov complexity $C^A$
allowing programs to access the oracle. Then $C^A(\alpha)$ is
defined in a natural way. The results of this comparison
are shown by a diagram (Fig.~\ref{t}).
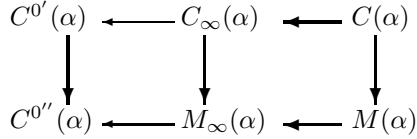
\begin{figure}[h]
\begin{center}
\setlength{\unitlength}{3ex}	
\begin{picture}(6,4)
\put(-3,0){$C^{0''}(\alpha )$}
\thicklines \put(-1.3,2.7){\vector(0,-1){2}}
            \put(2.7,2.7){\vector(0,-1){2}}
            \put(7.7,2.7){\vector(0,-1){2}} \thinlines
\put(-3,3){$C^{0'}(\alpha )$}
\put(1.8,3.1){\vector(-1,0){2.1}}
\thicklines \put(6.7,3.1){\vector(-1,0){1.7}} \thinlines
\put(1.8,0.1){\vector(-1,0){2.1}}
\thicklines \put(6.7,0.1){\vector(-1,0){1.7}} \thinlines
\put(2,3){$\KK(\alpha )$}
\put(2,0){$\MM(\alpha )$}
\put(7,3){$C(\alpha )$}
\put(7,0){$M(\alpha )$}
\end{picture}
\end{center}
\caption{Relations between different complexity
  measures for infinite sequences.
  Arrows go from the bigger quantity to the smaller one (up to
$\calo(1)$-term, as usual). Bold arrows indicate inequalities
that are immediate consequences of the definitions. Other arrows
are provided by \cite[Theorems~1 and 4]{bsv}.}\label{t}
\end{figure}

On this diagram no arrow could be inverted. We have mentioned this
for the rightmost four arrows. For the remaining three
arrows this is obvious.
Indeed,
$C^{0''}(\alpha)$ is finite while $C^{0'}(\alpha)$ is infinite for a sequence
$\alpha$ that is ${0}''$-computable but not
${0}'$-computable. Therefore the leftmost
downward arrow cannot be
inverted. The leftmost leftward arrows cannot be inverted for similar reasons:
$C^{0'}(\alpha)$ and
$C^{0''}(\alpha)$ are finite while $\KK(\alpha)$ and $\MM(\alpha)$ are infinite
for a sequence that is ${0}'$-computable but not
computable.

The statements we cited do not tell
us whether the inequality $C^{0'}(\alpha)\le \MM(\alpha)+\calo(1)$ is true
or not.
Another question left open in~\cite{bsv}
is the following: are the 
inequalities
$$
\KK(\alpha)\le \K^{0'}(\alpha)+O(1),\quad
\MM(\alpha)\le \K^{0''}(\alpha)+O(1)
$$
true on the domain of 
\emph{computable}
sequences? In this paper we answer the first question in positive
and the remaining two questions in negative (Theorems~\ref{th1} and \ref{th2} below).
Thus we get the following diagram
for complexities of computable sequences:
\begin{center}
\setlength{\unitlength}{4ex}	
\begin{picture}(6,4)
\put(-2,0){$\K^{0'}(\alpha)$}
\put(4.8,3.2){$\infty$}
\put(4.8,.2){$\infty$}
\put(.6,.2){$\infty$}
\put(6.8,1.7){$\infty$}
\put(2.8,1.7){$2$}
\put(2.7,2.7){\vector(0,-1){2}}
\put(6.7,2.7){\vector(0,-1){2}}
\put(5.8,3.1){\vector(-1,0){1.7}}
\put(1.8,0.1){\vector(-1,0){2.1}}
\put(5.8,0.1){\vector(-1,0){1.7}}
\put(2,3){$\KK(\alpha)$}
\put(2,0){$\MM(\alpha)$}
\put(6,3){$\K(\alpha)$}
\put(6,0){$M(\alpha)$}
\end{picture}
\end{center}

The sign $2$ near the arrow means that the larger quantity
is at most 2 times the smaller quantity (up to an additive constant),
and the sign $\infty$
means that the larger quantity cannot by bounded
by any computable function of the smaller quantity
even for computable sequences.

It is instructive to compare these  results with similar
results for finite sequences (i.e. strings).
For $x\in\{0,1\}^*$ let
$\MM(x)=\liminf C(x|n)$ and $\KK(x)=\min\{l(p)\mid p(n)=x
\text{ for almost all }n\}$.
From definitions it is straightforward that
$\MM(x)\le \KK(x)\le C^{0'}(x)$ for all $x$ (up to an additive constant).
And by~\cite{ver} 
we have $C^{0'}(x)\le \MM(x)+O(1)$, hence all the three
quantities coincide up to an additive constant.
Similar inequality holds for infinite sequences as well (Theorem~\ref{th1}
from the present paper). However, the analog of the straightforward inequality
$\MM(x)\le C^{0'}(x)+O(1)$ is not true for infinite sequences,
even on the domain of computable sequences.

\section{Theorems and proofs}

\begin{theorem}\label{th1}
   $C^{0'}(\alpha)\le M_{\infty}(\alpha)+O(1)$.
\end{theorem}  
\begin{proof}
  Fix $k$ and consider the set $S$
  of all binary strings $x$ with $C(x|l(x))\le k$.
  This set is computably enumerable uniformly on 
  $k$. The \emph{width} of $S$ is less than $2^{k+1}$
  (this means that for all $n$ the set contains less than
  $2^{k+1}$ strings of length $n$).

  We will view the set $\{0,1\}^*$
  of all binary strings as a rooted tree. Its root is the empty string $\Lambda$
  and each edge connects a vertex  $x$ with its \emph{children}
  $x0$ and $x1$.

\emph{An infinite path in} $S$ is an infinite sequence 
of vertices $x_0,x_1,x_2,\dots$ from $S$  such that 
$x_i$ is a child of $x_{i-1}$ for all $i>0$.
Let us stress that we do not require 
infinite paths start in the root, that is,
$x_0$ may be non-empty.

If $M_{\infty}(\alpha)\le k$, then for some
$n$ prefixes of $\alpha$ of length at least $n$
form an infinite path in $S$.\footnote{If, moreover,
$M(\alpha)\le k$, then 
that path starts in the root.}
We have to show that in this case
$C^{0'}(\alpha)\le k +O(1)$.

The proof will follow from
two lemmas. To state the lemmas we need yet another definition.
A set $T$ of strings is called \emph{leafless},
if for all $x\in T$ at least one child of $x$
is in $T$.

\begin{lemma}[on trimming leaves]\label{l-cut}
For every computably enumerable set 
$S\subset\{0,1\}^*$ of width at most 
$w$ there is a 
computably $0'$-decidable
set $T\subset\{0,1\}^*$
such that 
\begin{enumerate}
\item[(1)] $T$ is leafless,
\item[(2)] the width of $T$ is at most  $w$,
\item[(3)] $T$ includes all infinite paths in $S$.
\end{enumerate}
The program of
the algorithm that
$0'$-recognizes $T$
can be found from $w$ and
the program enumerating $S$.   
\end{lemma}

\begin{lemma}\label{l-col}
Let $T$ is a leafless set of width at most $w$.
The for any infinite 
 0-1-sequence $\alpha$ whose sufficiently large 
 prefixes form and infinite path in $T$
 we have $C^T(\alpha)\le \log w +O(1)$.
The constant $O(1)$ does not depend on $T,\alpha,w$.
Moreover, the program witnessing
the inequality $C^T(\alpha)\le \log w +O(1)$
needs only an oracle enumerating the set $T$ (in any order).
\end{lemma} 

We first finish the proof of the theorem
assuming the lemmas.
By applying Lemma~\ref{l-cut}
to the set $S=\{x\mid C(x|l(x))\le k\}$
we obtain a leafless set 
$T$ of width less than $2^{k+1}$ that includes
all infinite paths in $S$ and is $0'$-decidable
uniformly on $k$.
If $M_{\infty}(\alpha)\le k$, then   
by the second lemma $C^{T}(\alpha)\le k+O(1)$.
Since $T$ is $0'$-decidable
uniformly on $k$ and we can
retrieve $k$ from the length
of the program witnessing the inequality
$C^{T}(\alpha)\le k+O(1)$,
we can conclude that 
$C^{0'}(\alpha)\le k+O(1)$.

It remains to prove the lemmas.
We start with the proof of the simpler
Lemma~\ref{l-col}.

\begin{proof}[Proof of Lemma~\ref{l-col}]
  Basically we have to number
  infinite paths in $T$ in such a way that
  given the number of a path we can find
  all its vertices. We will imagine that we have
  \emph{tokens}  with numbers from $1$ to $w$,
  and move those tokens
  along infinite paths in $T$.
  The number of an infinite path in $T$
  will be the number of the token
  that moves along that path.

  More specifically, we start an enumeration of the set $T$.
  Observing string enumerated in $T$ we will place tokens
  on some of them; vertices baring tokens will be called
  \emph{distinguished}. We will
  do that so that the following be true:

  (1) distinguished vertices are pair wise inconsistent
  (neither of them is a prefix of another one),

  (2) every string enumerated so far in  $T$ is a prefix of some
  distinguished vertex,

  (3) tokens  move only from a vertex to its descendant
  (=extension).

  At the start no strings are enumerated
  so far and all tokens are not used. 
  When a new string $x$ is enumerated into  $T$,
  we first look whether it is a prefix of a distinguished vertex.
  In that case we do nothing, since property (2) remains true.

Otherwise property (2) has been violated. If 
$x$ is an extension of a distinguished string $y$
(such a vertex $y$ is unique by property (1)),
then we move the token from $y$ to $x$ 
keeping (1) and (3) true and
restoring (2).

Finally, if $x$ is inconsistent
with all distinguished nodes,
we take a new token and place it on $x$
restoring (2) and keeping (1).

Since $T$ is leafless and its width
is at most $w$,
the set $T$ cannot have more than $w$ pairwise 
inconsistent strings (for all
large enough $n$ each of those strings has a length-$n$
extension in $T$ and those extensions are pair wise different).
Therefore we do not need more than $w$ tokens.

By construction for every infinite path in $T$
a token is at certain time
placed on a vertex of the path
and moves along the
path infinitely long.

To every natural number $i$ from $1$ to $w$ we assign
a program $p_i$ that for input $n$ waits until the $i$th
token is placed on a string $x$ of length at least $n$,
then it prints the first  $n$ bits of  $x$.
\end{proof}

It remains to prove the first lemma.

\begin{proof}[Proof of Lemma~\ref{l-cut}]
  It seems natural to let $T$ be the union of all 
infinite paths in $S$. In this case the conditions 
(1)--(3) hold automatically. However, this set is only  
$\Pi_2$, since 
$$
T=\{x\mid \forall i \text{ there is an extension of $x$ of length 
   $l(x)+i$ in  $S$}\}.
$$
It is not hard to find an example of a c.e. set $S$ for which this set
$T$ is $\Pi_2$ complete (and hence is not $0'$-decidable).
The set $T$ we construct will be larger in general case than the 
union of all 
infinite paths in $S$.

We will be using Cantor topology
on the set of subsets of $\{0,1\}^*$.
Its base consists of sets of the form:
$$
\{X\subset \{0,1\}^*\mid A\subset X,\ B\cap X=\emptyset\},
$$
where $A,B$ are any \emph{finite}
subsets of $\{0,1\}^*$. Open sets in Cantor topology
are arbitrary unions of these sets.
It is well known that this  topological space is compact.

We will consider leafless sets $T$
such that the width of the set $T\cup S$ does not exceed
$w$.
Such sets will be called \emph{acceptable}.
For instance, the empty set is acceptable.
The key observation is the following: \emph{the family of
acceptable sets is closed in  Cantor topology}.

The set $T$ is defined as the largest acceptable set with respect to
some linear order. More specifically,
consider the lexicographical order
on binary strings (for strings of different length,
the shorter string is less than the longer one).
Then we define $X<Y$ for different sets $X,Y\subset \{0,1\}^*$
if the lex first string in the symmetric difference of
$X,Y$  belongs to $Y\setminus X$ (in other words,
we compare sets according to the lexicographical order on their characteristic
sequences).
Not every non-empty family of subsets of
$\{0,1\}^*$ has the largest set with respect to this order.
However, this holds for \emph{closed} families.
Hence there exists the largest acceptable set $T$.

In other words, one can define $T$ recursively:
enumerate all binary strings $x_1,x_2,\dots$ according to
the lexicographical order, then 
put $x_i$ in $T$ if there is an acceptable set $R$ which
includes the set
$T\cap\{x_1,\dots,x_{i-1}\}$, or, equivalently, 
$$
R\cap\{x_1,\dots,x_{i-1}\}=T\cap\{x_1,\dots,x_{i-1}\}.
$$
This definition guarantees that for all $i$
there is an acceptable set  $R$ with 
$R\cap\{x_1,\dots,x_{i-1}\}=T\cap\{x_1,\dots,x_{i-1}\}$.
Since the family of acceptable sets is closed,
this implies acceptability of $T$.
And by construction this $T$ is larger than 
or equal to every acceptable set.

Properties (1) and (2) hold automatically for $T$.
Let us verify the property (3).
Let $\alpha$ be an infinite path in  $S$.
Consider the set  $T'=T\cup\alpha$.
It is leafless  (since every  vertex from $\alpha$
has a child in
$\alpha$).
Besides, $T'\cup S=T\cup S$ (as  $\alpha\subset S$),
hence  $T'$ is acceptable.
The definition of $T$ implies that it is a maximal
w.r.t. inclusion acceptable set.
Therefore  $T'=T$, or, in other words,  $\alpha\subset T$.

It remains to show that
$T$ is $0'$-decidable. Assume that we already know for every
string among $x_1,\dots, x_{i-1}$ whether it belongs to $T$ or not.
We have to decide whether $x_i\in T$.
By construction  $x_i$ is in  $T$
if and only if there is an
acceptable set including the set $T\cap\{x_1,\dots, x_{i-1}\}$ and $x_i$.
Thus it suffices to prove that
for any finite $E\subset\{0,1\}^*$
we can decide with the help of $0'$ whether there is an
acceptable set including  $E$ or not.
To this end we reformulate this property of $E$.
Fix a computable enumeration
of $S$ and denote by $S_j$ the subset of $S$ consisting of all
strings enumerated in $j$ steps.

Call a set  $R$ \emph{acceptable at time $j$} if
it is leafless and the width of  $R\cup S_j$ is at most  $w$.
We claim that
\begin{quote}
there is an acceptable set including  $E$
if and only if
for all $j$ there is a set $R_j$ including  $E$
that is acceptable at time $j$.
 \end{quote}
Since acceptability implies acceptability at time $j$ for all $j$,
one direction is straightforward.
In the other direction: assume that for every
$j$ there is a set
$R_j \supset E$
which is acceptable at time
$j$. We have to construct an 
acceptable set $R\supset E$.

By compactness arguments, the sequence
 $R_1,R_2,\dots$
has an accumulation point $R$.
Since both properties
``to include $E$'' and ``be leafless'' are closed,
the set $R$ possesses these properties.
It remains to show that
the width of the set $R\cup S$ is at most $w$.
For the sake of contradiction
assume that there are $w+1$ strings of the same length $n$ that
belong to $R\cup S$.
Then consider the 
(open) family that consists 
of all sets $R'$ such that 
the set  $R'\cup S$ includes all those strings.
Since $R$ is in this family, for infinitely many  $j$
the set $R_j$ is in this family.
Choose such a $j$ for which
 $S_j$ includes all those strings. We obtain a contradiction, as 
the width of the set                
$R_j\cup S_j$ is at most $w$.

It remains to show decidability of the following property of the pair
 $E,j$:
\emph{there is a set $R$ including  $E$
that is acceptable at time $j$} 
(indeed, in this case the oracle
$0'$ is able to decide whether this property holds
for all  $j$).
Indeed, the sets $S_j$ and $E$ are finite. Let $n$ be the maximal
length of strings from these sets. Without
loss of generality we may assume  
that each string $x\in R$ of length $n$ or larger
has exactly one child in $R$, namely,  $x0$,
and all strings from $R$ of length larger than $n$
are obtained from strings of length $n$ from $R$ by appending zeros.
Such sets $R$ are essentially finite objects
and there are finitely many of them.
For any such set we can decide whether it
includes  $E$ and is acceptable at time $j$.
The lemma is proved.   
    \end{proof}

\begin{remark}
  The set $T$ constructed in the proof of Lemma 2
  can be defined in several ways.
In the original proof,
it was defined as the limit of the sequence
$R_1,R_2,\dots$ where $R_j$ is the largest set
that is acceptable at time $j$.
One can show that this sequence has a limit indeed.
So defined, $T$ is obviously 
$0'$-decidable.
B. Bauwens suggested  another way to define (the same)
set $T$:
include $x_i$ in $T$ if for all $j$
there is a set that is acceptable at time
$j$ and includes $T\cap\{x_1,\dots, x_{i-1}\}$ and $x_i$.
Again, so defined $T$ is obviously $0'$-decidable.
In the above proof, we defined
$T$ in a way that is independent
on the chosen enumeration of the set $S$.
This construction of $T$  simplifies the verification
of properties (1)--(3), but proving $0'$-decidability of $T$
becomes harder.
\end{remark}
\end{proof}

Now we know that 
 $C^{0'}(\alpha)\le M_{\infty}(\alpha)+O(1)$.
How large can be the gap between
$C^{0'}(\alpha)$ and 
$M_{\infty}(\alpha)$?
For $\alpha$ equal to the characteristic sequence of $0'$
the gap is infinite,
since $C^{0'}(0')$ is finite while
$M_{\infty}(0')$ is infinite.
However, we are mostly interested in computable sequences,
thus we refine the question:
How large can be the gap between
$C^{0'}(\alpha)$ and 
$M_{\infty}(\alpha)$ for \emph{computable} sequences $\alpha$?

It turns out that this such gap can be arbitrary large:
$M_{\infty}(\alpha)$ cannot be bounded 
by any computable function of $C^{0'}(\alpha)$.
More specifically, the following holds:
\begin{theorem}\label{th2}
  For any computable function
  $f:\N\to \N$ for all  $m$ there is a  \emph{computable} 
  sequence $\alpha$ for which  $C^{0'}(\alpha)\le m+O(1)$
while $M_{\infty}(\alpha)\ge f(m)$. The constant $O(1)$
depends on the function $f$.
\end{theorem}

\begin{proof}
  A natural approach to construct such
  sequence $\alpha$ is to take
  a sufficiently long prefix of
  $0'$ and extend it by zeros.
  More specifically, let $x$ stand
  for a prefix encoding of the first $m$ bits
  of $0'$, say $x=0^m1{0'}_{1:m}$, and let $\alpha=x000\dots$.
  This approach fails, as
  whatever $m$ we choose the complexities
  $C^{0'}$ and    $M_{\infty}$ of this
  sequence coincide up to an additive constant.
  Indeed, $C^{0'}(\alpha)\ge C^{0'}(m)-O(1)$,
  since from $\alpha$ we can find $m$.
  On the other hand, $M_{\infty}(\alpha)\le C^{0'}(m)+O(1)$:
  pick a program $p$ with oracle $0'$
  whose length is $C^{0'}(m)$ and that prints $m$.
Assume that
  $n$  is larger than the number of steps
  needed to enumerate all numbers at most
  $m$ into $0'$ and is larger than all queries by $p$ to its oracle.
  Then we can find $x$ from
  $n$ and $p$: first make $n$ steps of enumerating
  $0'$ and run $p$ with the subset $A$ of $0'$ we have obtained
  instead of the full oracle $0'$. The program $p$ will print $m$. 
  Then we find $x$, as the length-$m$ prefix of the characteristic
  sequence of $A$ and output $\alpha_{1:n}$.

  To prove the theorem we will use the Game Approach.
Assume that a natural parameter $w$ is fixed.
Consider the following game between two players, Alice and Bob.
Players turn to move alternate.
On each move each player can paint any string  or do nothing.
We will imagine that Alice uses green color and Bob uses red color
(each string can be painted in both colors).
For every $n$ Alice may paint at most  $w$ strings of length 
$n$. The player make infinitely many
moves and then the game ends.
Alice wins if (1) for some $n$
there are $w$ strings of length $n$ who all have been painted by both players
($w$ red-green strings of the same length), or (2)
there is an infinite 0-1-sequence $\alpha$
such that $\alpha_{1:n}$ is 
the lex first green string of length $n$ for all $n$
and $\alpha_{1:n}$ has not been painted by Bob (is not red)
for infinitely many $n$.

From the rules of the game it is clear that it does not matter
who starts the game (postponing a move does not hurt).

\begin{lemma}
  For every $w$ Alice has a winning strategy
  in this game and this strategy is computable uniformly on $w$.
\end{lemma}
\begin{proof}
  Alice's strategy is recursive.
  If $w=1$, then Alice just paints the strings
  $\Lambda,0,00,000,\dots$ (on her $i$th move she paints
  the string $0^i$).

  Assume now that we have already defined Alice's winning strategy
  in the $w$-game,   we will call it the \emph{$w$-strategy}.
Then Alice can win $(w+1)$-game as follows:
she paints first the empty string
and then runs the $w$-strategy in the subtree
with the root 1.\footnote{Formally, that means that Alice adds prefix 1 to
every move made by the $w$-strategy, postpones
Bob's moves that do not start with 1, and for every Bob's
move of the form $1x$ tells the $w$-strategy that Bob has made the move $x$.}
If $w$-strategy wins in the first  way
(that is, for some  $n$ there are $w$ red-green
strings of length $n$ that start with 1),
then Alice stops  $w$-strategy.
She then paints the strings 
$0,00,000,\dots,0^m$ where $m$ is larger
than the length of all strings painted by $w$-strategy.
Then Alice runs $w$-strategy for the second time, but this time in the
subtree with the root $0^m1$. Again, if the second run of
$w$-strategy wins in the first  way,
then Alice stops it
and paints the strings
$0^{m+1}, 0^{m+2},\dots, 0^l$  where $l$ is larger
than the length of all strings painted green so far. And so on.

The $w+1$-strategy is constructed. Let us show that it obeys the rules,
that is, for all $n$ it paints at most
$w+1$ strings of length $n$. Indeed, for each $n$
at most  $w$ strings of length $n$ were painted by 
a run of $w$-strategy (different runs of $w$-strategy paint strings
of different lengths) and besides the string $0^n$ might be painted.

Let us show that $w+1$-strategy wins the game. We will distinguish
two cases.

\emph{Case 1.} We have run $w$-strategy infinitely many times.
Then each its run has won in the first way.
Hence  for infinitely many $n$ there exist 
$w$ red-green strings of length $n$ and all those strings
have a 1 (indeed, we have run $w$-strategies only
in subtrees with roots of the form $00\dots01$).
Consider now the strings $0^n$ for those $n$'s.
If at least one of them has been painted by Bob,
then we have won in the first way.
Otherwise the nodes $\Lambda,0,00,\dots$
are lex first green strings of lengths  $0,1,2,\dots$
and infinitely many of them are not red.
This means that we have won in the second way.

\emph{Case  2.} A run of $w$-strategy, say in the subtree
with root $0^l1$,
has not been stopped
and hence it won in the second way.
Then the string $0^l1$ has been extended by an infinite
green path $P$ (including the string  $0^l1$ itself) that
contains infinitely many non-red nodes.
Since, all the nodes 
$0,00,000,\dots,0^l$ are also green, the path
$0^l1P$ consists entirely of green nodes, starts in the root,
contains infinitely many non-red nodes and
all its nodes are lex first green nodes 
(recall that all strings wit prefix $0^l0$ have not been
painted by Alice).
\end{proof}

To prove the theorem we apply 
$2^{f(m)}$-strategy against the following ``blind''
Bob's strategy: Bob paints a string $x$ of length $n$
when he finds a program  
$p$ of length less than $f(m)$ with $p(n)=x$ (he
runs all programs of length  less than $f(m)$
on all inputs in a dovetailing style).
This strategy is computable and for all
$n$ it paints less than $2^{f(m)}$ strings  of
length $n$. Hence Alice wins in the second way:
there is an infinite green
path whose infinitely many nodes are not red.
Call this path $\alpha$. By construction
$M_{\infty}(\alpha)\ge f(m)$.

On the other hand, the set
of all green nodes is computably enumerable
and its width is at most 
$f(m)$. Hence $M(\alpha)<f(m)+O(1)$ and
by Meyer's theorem $\alpha$ is computable.

Finally, the path $\alpha$ can by computed from $m$
with oracle $0'$: for every $n$
we can find the lex first green string for length $n$
Hence $C^{0'}(\alpha)<\log m+O(1)<m+O(1)$.
The theorem is proved.
\end{proof}

\section*{Acknowledgments}
The author is sincerely grateful to Bruno Bauwens
for providing an alternative proof of Lemma~\ref{l-cut}.

\end{document}